\documentclass[11pt,a4paper]{article}

\usepackage{epsf,epsfig,amsfonts,amsgen,amsmath,amstext,amsbsy,amsopn,amsthm,cases,listings,color
}
\usepackage{ebezier,eepic}
\usepackage{color}
\usepackage{multirow}
\usepackage{epstopdf}
\usepackage{graphicx}
\usepackage{pgf,tikz}
\usepackage{mathrsfs}
\usepackage[marginal]{footmisc}
\usepackage{enumitem}
\usepackage[titletoc]{appendix}
\usepackage{booktabs}
\usepackage{url}
\usepackage{mathtools}

\usepackage{pgfplots}
\usepackage{authblk}
\usepackage{amssymb}

\usepackage{wasysym}

\usepackage{empheq}

\usepackage{dsfont}

\usepackage{subfigure}
\pgfplotsset{compat=1.18}
\usepackage{mathrsfs}
\usepackage{wasysym}
\usetikzlibrary{arrows}
\usepackage{aligned-overset}

 \usepackage[backref=page]{hyperref}

\allowdisplaybreaks[1]

\definecolor{uuuuuu}{rgb}{0.27,0.27,0.27}
\definecolor{sqsqsq}{rgb}{0.1255,0.1255,0.1255}

\newtheorem{definition}{Definition} [section]
\newtheorem{theorem}[definition]{Theorem}
\newtheorem{lemma}[definition]{Lemma}

\newtheorem{claim}[definition]{Claim}

\newenvironment{pf}{\noindent{\bf Proof}}

\newcommand{\uproduct}{\mathbin{\;{\rotatebox{90}{\textnormal{$\small\Bowtie$}}}}}

\begin{document}
\title{\bf\Large Counting triangles in graphs without vertex disjoint  odd cycles}
\date{ }
\author[1]{Jianfeng Hou\thanks{Research was supported by National Natural Science Foundation of China (Grant No. 12071077). Email: jfhou@fzu.edu.cn}}
\author[1]{Caihong Yang\thanks{Email: chyang.fzu@gmail.com}}
\author[1]{Qinghou Zeng\thanks{Research was supported by National Natural Science Foundation of China (Grant No. 12001106, 12371342) and National Natural Science Foundation of Fujian Province (Grant No. 2021J05128). Email: zengqh@fzu.edu.cn}}
\affil[1]{Center for Discrete Mathematics,
            Fuzhou University, Fujian, 350003, China}


\maketitle
\begin{abstract}
Given two graphs $H$ and $F$, the maximum possible number of copies of $H$ in an $F$-free graph on $n$ vertices is denoted by $\mathrm{ex}(n, H, F)$. Let $(\ell+1) \cdot F$ denote $\ell+1$ vertex disjoint copies of $F$. In this paper, we determine the exact value of $\mathrm{ex}(n, C_3, (\ell+1)\cdot C_{2k+1})$ and its extremal graph, which generalizes some known results.

\medskip

\end{abstract}
\section{Introduction}\label{SEC:Introduction}

Given a graph $F$ we say a graph $G$ is $F$-free if it does not contain $F$ as a subgraph. The {\em Tur\'{a}n number} $\mathrm{ex}(n,F)$ of $F$ is the maximum
number of edges in an $F$-free graph on $n$ vertices.  The study of $\mathrm{ex}(n,F)$ is perhaps the central topic in extremal graph,
and one of the most famous results in this regard is Tur\'{a}n's theorem \cite{41Turan}, which states that for every integer $r \ge 2$ the Tur\'{a}n number $\mathrm{ex}(n,K_{r+1})$ is uniquely achieved by the complete balanced $r$-partite graph on $n$ vertices, which is called the \emph{Tur\'{a}n graph} $T_{r}(n)$. Starting from this, the Tur\'{a}n problem has attracted much attention. For more problems and results in this area, we refer the reader to \cite{BJ17, Er62, FG15, HHLYZ23, Moon68}.

A generalization of Tur\'{a}n numbers was introduced by Alon and Shikhelman \cite{AS16}. Given two graphs $T$ and $F$, the {\em generalized Tur\'{a}n number} of $T$ and $F$, denoted by $\mathrm{ex}(n, T, F)$, is the maximum possible number of copies of $T$ in an $F$-free graph on $n$ vertices. We use $\mathrm{EX}(n,T, F)$ to denote the collection of all $n$-vertex $F$-free graphs with exactly $\mathrm{ex}(n, T, F)$ copies of $T$. This topic was first studied by Erd\H{o}s \cite{E62}, who determined $\mathrm{ex}(n, K_t, K_r)$ for all $t<r$. For this kind of problem, several authors (e.g., see \cite{FO17, G2023, GGMV20, liu2023, MQ2020, ZC22, ZC23}) have given many nice proofs for generalized Tur\'{a}n problem.

A typical example is to determine $\mathrm{ex}(n, T, F)$ when both $T$ and $F$ are odd cycles. For $k\ge 3$, let $C_k$ denote a cycle on $k$ vertices. In 1984, Erd\H{o}s \cite{E84} conjectured that $\mathrm{ex}(n, C_5, C_3) \leq (n/5)^5$. Note that when $n$ is divisible by $5$, the balanced blow-up of a $5$-cycle shows that $\mathrm{ex}(n, C_5, C_3) \geq (n/5)^5$. Using the method of flag algebras, Erd\H{o}s conjecture was proved by Grzesik \cite{G12} and independently by Hatami, Hladk\'{y}, Kr\'{a}l', Norine and Razborov \cite{HHK13}.

Motivated by Erd\H{o}s conjecture \cite{E84}, Bollobós and Gy\H{o}ri \cite{BG08} initiated the study of the natural converse of this problem by proving
\[
(1+o(1))\frac{n^{3/2}}{3\sqrt{3}} \leq \mathrm{ex}(n, C_3, C_5) \leq (1+o(1))\frac{5}{4}n^{3/2}.
\]
Gy\H{o}ri and Li \cite{GL13} obtained results bounding $\mathrm{ex}(n, C_3, C_{2k+1})$ which were improved by Alon and Shikhelman \cite{AS16}.

In this paper, we count the number of triangles in graphs without vertex disjoint copies of odd cycles. For an integer $\ell \ge 1$,  let $\ell \cdot F$ denote $\ell$ vertex disjoint copies of $F$.  Given two graphs $G$ and $H$ whose vertex sets are disjoint, we define the \emph{join} $G \uproduct H$ of $G$ and $H$ to be the graph obtained from $G \sqcup H$ (the vertex-disjoint union of $G$ and $H$) by adding all edges that have nonempty intersection with both $V(G)$ and $V(H)$. Obviously,  $K_{\ell} \uproduct T_{r}(n-\ell)$ is $(\ell+1)\cdot K_{r+1}$-free. Erd\H{o}s~\cite{Er62} showed that if $n$ is sufficiently large, then $K_{\ell} \uproduct T_{2}(n-\ell)$  is the unique extremal graph for $(\ell+1)\cdot K_{3}$. It was extended to $(\ell+1)\cdot K_{r}$ for all $r \ge 3$ by Moon~\cite{Moon68}. For general graph $F$, Gorgol showed that $\mathrm{ex}(n, \ell \cdot F)=\mathrm{ex}(n, F)+O(n)$. Recently, Hou, Li, Liu, Yuan and Zhang\cite{HHLYZ23} present a general approach for determining $\mathrm{ex}(n, \ell \cdot F)$ especially when $F$ is a  hypergraph.

For an integer $k\ge 2$, $K_{\ell} \uproduct T_{2}(n-\ell)$ is $(\ell+1)\cdot C_{2k+1}$-free. Counting triangles in such a graph yields that  $\mathrm{ex}(n, C_3, (\ell+1)\cdot C_{2k+1})=\Omega(n^2)$. Thus, unlike the case of ordinary Tur\'{a}n problems, $\mathrm{ex}(n, C_3, C_{2k+1})$ and $\mathrm{ex}(n, C_3, (\ell+1)\cdot C_{2k+1})$ can have different order of magnitudes. For the special case that $\ell=1$ and $k=2$, Gerbner, Methuku and Vizer \cite{GMV19} showed $\mathrm{ex}(n, C_3, 2\cdot C_5)=\Theta(n^2)$.  Recently, Zhang, Chen, Gy\H{o}ri and Zhu \cite{ZCG23} determine the exact value of $\mathrm{ex}(n, C_3, 2\cdot C_5)$. In this paper, we extend it  and  determine  $\mathrm{ex}(n, C_3, (\ell+1) \cdot C_{2k+1}) $ by showing the following theorem.                                                                                                                                                                                                                                                                                                                                                                                                                                                                                                                                                                                                                                                                                                                                                      

\begin{theorem}\label{main-theorem}
For integers $\ell \geq 1$ and $k\ge 1$, we have
\[
\mathrm{ex}(n, C_3, (\ell+1) \cdot C_{2k+1}) = \ell \left\lfloor\frac{(n-\ell)^2}{4}\right\rfloor + (n-\ell)\binom{\ell}{2} + \binom{\ell}{3}
\]
for sufficiently large $n$,  and $K_{\ell} \uproduct T_2(n-\ell)$ is the unique extremal graph.
\end{theorem}

This paper is organized as follows. In Section \ref{SEC:P}, we give a lemma about the sum of edges and triangles in $C_{2k+1}$-free graphs which plays a key role in the proof of Theorem \ref{main-theorem}. In Section \ref{SEC:pf-main-thm}, we prove Theorem \ref{main-theorem}.


\section{Key lemma}\label{SEC:P}


In this section, we give a lemma about the sum of edges and triangles in $C_{2k+1}$-free graphs which plays a key role in the proof of Theorem \ref{main-theorem}.
At first, we describe notations and  terminologies.  Let $G$ be a graph, we denote its vertex set by $V(G)$ and the edge set by $E(G)$, and let $v(G) = |V(G)|$, $e(G) = |E(G)|$. For two subsets $S, T$ of $V(G)$, we use $G[S]$ to denote the subgraph of $G$ induced by $S$. Let $N(v)$ denote the neighbours of $v$ in $G$ and $N_S(v) = N(v) \cap S$. We use $G[S, T ]$ to denote the bipartite subgraph induced by the edges with one end in $S$ and the other in $T$. Let $e(S) = e(G[S])$ and $e(S, T ) = e(G[S, T ])$. Denote by $G - S$ the subgraph of $G$ induced by $V(G)\setminus S$. When $S = \{v\}$, we simply write $G - v$. For any $v \in V(G)$, let $d_G(v)$ be the degree of $v$ in $G$ and $d_S(v)$ be the number of neighbours of $v$ in $S$. We use  $t(G)$ to  denote the number of triangles in $G$. For $k\ge 1$, let $P_k$ denote a path on $k$ vertices. A $3$-uniform hypergraph $\mathcal{H}$ is a collection of $3$-subsets of some finite set $V$. For a vertex $v\in V(\mathcal{H})$, the \emph{link} $L_{\mathcal{H}}(v)$ of $v$ in $\mathcal{H}$ is a graph on $V(\mathcal{H})$ and
\[
L_\mathcal{H}(v)=\left\{A \in \binom{V(\mathcal{H})}{2} \colon A \cup \{v\}\in \mathcal{H}\right\}.
\]
The \emph{degree} of $v$ in $\mathcal{H}$ is $d_{\mathcal{H}}(v)=e(L_\mathcal{H}(v))$. Denote by $\Delta(\mathcal{H})$
maximum degree of $\mathcal{H}$.

Now we list some  results which will be used in our proofs.  The first one is the well-known Erd\H{o}s-Gallai theorem about paths \cite{EG59}.
\begin{theorem}[Erd\H{o}s and Gallai \cite{EG59}]\label{59EG}
For positive integers $n$ and $k$ with $n \geq k \geq 2$,
\begin{align*}
\mathrm{ex}(n, P_k) \leq \frac{1}{2}(k-2)n.
\end{align*}
\end{theorem}

The Tur\'{a}n number of odd cycles can be determined precisely. The following theorem was given by F\"{u}redi and Gunderson \cite{FG15}.

\begin{theorem}[F\"{u}redi and Gunderson \cite{FG15}]\label{15FG}
For positive integers $k\ge 1$ and $n \geq 4k-2$,
\begin{align*}
\mathrm{ex}(n, C_{2k+1}) = \left\lfloor\frac{n^2}{4}\right\rfloor.
\end{align*}
\end{theorem}

Determining the Tur\'{a}n number of even cycles is a key problem in extremal graph theory and has attracted much attention. The following result is needed.

\begin{theorem}[Bondy and Simonovits \cite{BS74}]\label{17BJ}
For positive integers $n$ and $k$,
\begin{align*}
\mathrm{ex}(n, C_{2k}) \leq 100kn^{1+1/k}.
\end{align*}
\end{theorem}

Determining the number of triangle in $C_{\ell}$ is closely related to the Tur\'{a}n number of even cycles. When $\ell$ is odd, Alon and Shikhelman \cite{AS16} proved the following theorem.

\begin{theorem}[Alon and Shikhelman \cite{AS16}]\label{16AS}
For positive integers $n$ and $k \geq 2$,
\[
\mathrm{ex}(n, C_3, C_{2k+1}) \leq \frac{16(k-1)}{3}\mathrm{ex}\left(\left\lceil \frac{n}{2} \right\rceil, C_{2k}\right).
\]
\end{theorem}

We also need a stability result of odd cycles which was given by Roberts and Scott \cite{RS18}. Let $F$ be a graph. We use  $\chi(F)$ to denote the chromatic number of $F$. If there is an edge $e$ in $F$ such that $\chi(F - e) < \chi(F)$, then $F$ is called \emph{edge-critical}.

\begin{theorem}[Roberts and Scott\cite{RS18}]\label{18RS}
Let $F$ be a edge-critical  graph with  $\chi(F) = k+1 \leq 3$, and let $g(n) = o(n^2)$ be a function. If $G$ is an $F$-free graph on  $n$ vertices with $e(G) \geq t_k(n) - g(n)$, then $G$ can be made $k$-partite by deleting $O(n^{-1}g(n)^{3/2})$ edges.
\end{theorem}

The following lemma is useful to prove Theorem \ref{main-theorem}.
\begin{lemma}\label{auxilary theorem}
For an integer $k\ge 1$, let $G$ be a $C_{2k+1}$-free graph on $n$ vertices. For sufficiently large $n$, we have
\begin{align*}
  e(G) + t(G) & \leq \left\lfloor \frac{n^2}{4}\right\rfloor,
\end{align*}
and the equality holds if and only if $G = T_2(n)$.
\end{lemma}
\begin{proof}
Let $G$ be a $C_{2k+1}$-free graph on $n$ vertices. It suffices to assume that $k\ge2$ and $G$ contains a triangle in view of Tur\'an's theorem. By Theorems \ref{17BJ} and \ref{16AS}, there exists a constant $c_0>0$ such that $t(G)\le c_0n^{1+1/k}$. Thus, we may also assume that
\[
e(G)\ge\left\lfloor \frac{n^2}{4}\right\rfloor-c_0n^{1+1/k}.
\]
Define
\[
f(G)=e(G) + t(G) - \left\lfloor \frac{n^2}{4}\right\rfloor  \quad\text{and}\quad  f(n)=\max_{|V(G)|=n}f(G).
\]
Note that $T_2(n)$ shows that $f(n)\ge 0$ for any integer $n\ge 1$.

Let $G_n$ be a $C_{2k+1}$-free graph on $n$ vertices with $e(G_n)\ge\left\lfloor n^2/4\right\rfloor-c_0n^{1+1/k}$ and $f(G_n)=f(n)$.  By Theorem \ref{18RS}, there exists a constant $c_1>0$ and a partition $V(G_n) = V_1 \cup V_2$ such that
\begin{align}\label{equ: lower bound of e(V1 V2)}
e(V_1, V_2) &\geq \frac{n^2}{4} - c_0n^{1+1/k} - c_1n^{(k+1)/(2k)} \geq \frac{n^2}{4} - 2c_0n^{1+1/k}.
\end{align}
Next our goal is to prove that each $G_n$ is $C_3$-free for sufficiently $n$. We next claim that if $G_n$ contains a triangle, then we can find a vertex $v$ in $G_n$ such that $d(v) < \frac{n}{2k} - 1$ and
\begin{align} \label{fn}
f(n-1) - f(n) &> k-1.
\end{align}

\begin{claim}
For each $i \in [2]$, $||V_i| - \frac{n}{2}| \leq (2c_0)^{1/2}n^{(1+k)/(2k)}$.
\end{claim}
\begin{proof}
Let $|V_1| = \frac{n}{2} + x$ and $|V_2| = \frac{n}{2} - x$. Then
\begin{align*}
e(V_1, V_2) &\leq |V_1||V_2| = \frac{n^2}{4} - x^2.
\end{align*}
Combining the inequality \eqref{equ: lower bound of e(V1 V2)}, we have $|x| \leq (2c_0)^{1/2}n^{(1+k)/(2k)}$.
\end{proof}
For each $i\in [2]$, define
\[
L_i := \left\{v\in V_i|d_{V_{3-i}}(v) \geq \left(1-\frac{1}{10k}\right)|V_{3-i}|\right\}
\]
and $S_i = V_i\backslash L_i$.
\begin{claim}
We have $|L_i| \geq \left(1- \frac{1}{4k}\right)|V_i|$ for $i \in [2]$.
\end{claim}
\begin{proof}
Observe that
\begin{align*}
e(V_1, V_2) &\leq |L_i||V_{3-i}| +\left(1-\frac{1}{10k}\right)|V_{3-i}||S_i|\\
&\leq |L_i||V_{3-i}|+ \left(1-\frac{1}{10k}\right)|V_{3-i}|(|V_i| - |L_i|)  \\
&= \left(1-\frac{1}{10k}\right)|V_i||V_{3-i}| + \frac{|L_i||V_{3-i}|}{10k}.
\end{align*}
This together with the inequality \eqref{equ: lower bound of e(V1 V2)} implies that there exists an integer $n_1$ such that $|L_i| \geq \left(1- \frac{1}{4k}\right)|V_i|$ for all $n \geq n_1$.
\end{proof}
Now we will show the following claim which will be used frequently.
\begin{claim}\label{fact common}
For any two vertices $u, v$ in $L_i$, $|N_{L_{3-i}}(u) \cap N_{L_{3-i}}(v)| \geq 2k$ for $i \in [2]$.
\end{claim}
\begin{proof}
Suppose that $u$ and $v$ are two different vertices in $L_i$. It is easy to see that
\begin{align*}
 |N_{L_{3-i}}(u) \cap N_{L_{3-i}}(v)|&\geq d_{L_{3-i}}(u) + d_{L_{3-i}}(v) - |L_{3-i}| \\
 &\geq 2\left((1-\frac{1}{10k}))|V_{3-i}| - S_{3-i}\right) - |L_{3-i}| \\
 &\geq 2k.
\end{align*}
As desired.
\end{proof}

\begin{claim}\label{independent set}
$L_i$ is an independent set for $i\in[2]$.
\end{claim}
\begin{proof}
Without loss of generality, we let $i = 1$. Suppose that $x_1x_2$ is an edge in $G_n[L_1]$. By $d_{V_2}(v) \geq \left(1-\frac{1}{10k}\right)|V_2|$, we can find two different vertices $y_1$ and $y_2$ in $L_2$ such that $x_1$, $x_2$ are adjacent to $y_1$, $y_2$, respectively. We can find two different vertices $x_3$ and $x_4$ in $L_1$ such that $x_3$, $x_4$ are adjacent to $y_1$, $y_2$, respectively. Step by step, repeating the above operation and choose a common neighbor $x_{k+1}$ of $y_{k-1}$ and $y_k$ (or the common neighbor $y_k$ of $x_k$ and $x_{k+1}$) at last step by Claim \ref{fact common}. Then we can get a copy of $C_{2k+1}$: $x_1y_1x_3y_3 \cdots y_kx_{k+1}y_{k-2}x_{k-2} \cdots y_2x_2x_1$, a contradiction.
\end{proof}
Recall that $G_n$ contains a triangle. Let $v_1v_2v_3$ is a triangle in $G_n$. Next we claim that there exists a vertex $v$ in $G_n$ such that $d(v) < \frac{n}{2k} - 1$. Let $\mathcal{L} = L_1 \cup L_2$ and $\mathcal{S} = S_1 \cup S_2$. In what follows, we prove this by considering three situations according to Claim \ref{independent set}, i.e.
\begin{enumerate}[label=(\roman*)]
\item\label{lab:1}
$v_1 \in S_1$, $v_2 \in L_1$, $v_3 \in L_2$.
\item\label{lab:2}
$v_2, v_3 \in \mathcal{S}$, $v_1 \in \mathcal{L}$.
\item\label{lab:3}
$v_1, v_2, v_3 \in \mathcal{S}$.
\end{enumerate}
For \ref{lab:1}, next we aim to find a path $y_1x_2y_3x_4 \cdots y_2x_1$ in $G[L_1, L_2]$ such that $y_1$ is adjacent to $v_2$, $x_1$ is adjacent to $v_3$. This forms a copy of $C_{2k+1}$ in $G_n$, a contradiction with the assumption of $G_n$. Let $x_j \in L_1$ and $y_j \in L_2$ for $j \in [k-1]$. Since $d_{V_{3-i}}(v) \geq \left(1-\frac{1}{10k}\right)|V_{3-i}|$ for every $v \in L_i$, we can find $x_1 \in N_{L_1}(v_3)$ and $y_1 \in N_{L_2}(v_2)$. We search for vertices one by one from the neighbours of the found vertices. At the last step, by Claim \ref{fact common} we choose the common neighbour $x_{k-1}$ of $y_{k-2}$ and $y_{k-1}$. Then $v_1v_2y_1x_2y_3x_4 \cdots y_2x_1v_3$ is a copy of $C_{2k+1}$ (see Figure 1). This is a contradiction. So it is impossible that $v_1v_2v_3$ is a triangle.
\begin{center}

\tikzset{every picture/.style={line width=0.75pt}} 

\begin{tikzpicture}[x=0.75pt,y=0.75pt,yscale=-1,xscale=1]

\draw   (181.2,20) .. controls (181.2,15.58) and (184.78,12) .. (189.2,12) -- (522,12) .. controls (526.42,12) and (530,15.58) .. (530,20) -- (530,44) .. controls (530,48.42) and (526.42,52) .. (522,52) -- (189.2,52) .. controls (184.78,52) and (181.2,48.42) .. (181.2,44) -- cycle ;
\draw   (181.2,94) .. controls (181.2,89.58) and (184.78,86) .. (189.2,86) -- (522,86) .. controls (526.42,86) and (530,89.58) .. (530,94) -- (530,118) .. controls (530,122.42) and (526.42,126) .. (522,126) -- (189.2,126) .. controls (184.78,126) and (181.2,122.42) .. (181.2,118) -- cycle ;
\draw    (230.2,11.6) -- (230.2,52.6) ;
\draw    (231.2,86.6) -- (231.2,126.6) ;
\draw  [fill={rgb, 255:red, 0; green, 0; blue, 0 }  ,fill opacity=1 ] (220,35.5) .. controls (220,34.12) and (218.88,33) .. (217.5,33) .. controls (216.12,33) and (215,34.12) .. (215,35.5) .. controls (215,36.88) and (216.12,38) .. (217.5,38) .. controls (218.88,38) and (220,36.88) .. (220,35.5) -- cycle ;
\draw    (217.5,35.5) -- (247.5,107.5) ;
\draw  [fill={rgb, 255:red, 0; green, 0; blue, 0 }  ,fill opacity=1 ] (280,107.5) .. controls (280,106.12) and (278.88,105) .. (277.5,105) .. controls (276.12,105) and (275,106.12) .. (275,107.5) .. controls (275,108.88) and (276.12,110) .. (277.5,110) .. controls (278.88,110) and (280,108.88) .. (280,107.5) -- cycle ;
\draw  [fill={rgb, 255:red, 0; green, 0; blue, 0 }  ,fill opacity=1 ] (310,107.5) .. controls (310,106.12) and (308.88,105) .. (307.5,105) .. controls (306.12,105) and (305,106.12) .. (305,107.5) .. controls (305,108.88) and (306.12,110) .. (307.5,110) .. controls (308.88,110) and (310,108.88) .. (310,107.5) -- cycle ;
\draw  [fill={rgb, 255:red, 0; green, 0; blue, 0 }  ,fill opacity=1 ] (340,107.5) .. controls (340,106.12) and (338.88,105) .. (337.5,105) .. controls (336.12,105) and (335,106.12) .. (335,107.5) .. controls (335,108.88) and (336.12,110) .. (337.5,110) .. controls (338.88,110) and (340,108.88) .. (340,107.5) -- cycle ;
\draw  [fill={rgb, 255:red, 0; green, 0; blue, 0 }  ,fill opacity=1 ] (280,35.5) .. controls (280,34.12) and (278.88,33) .. (277.5,33) .. controls (276.12,33) and (275,34.12) .. (275,35.5) .. controls (275,36.88) and (276.12,38) .. (277.5,38) .. controls (278.88,38) and (280,36.88) .. (280,35.5) -- cycle ;
\draw  [fill={rgb, 255:red, 0; green, 0; blue, 0 }  ,fill opacity=1 ] (310,35.5) .. controls (310,34.12) and (308.88,33) .. (307.5,33) .. controls (306.12,33) and (305,34.12) .. (305,35.5) .. controls (305,36.88) and (306.12,38) .. (307.5,38) .. controls (308.88,38) and (310,36.88) .. (310,35.5) -- cycle ;
\draw    (307.5,35.5) -- (277.5,107.5) ;
\draw    (307.5,35.5) -- (337.5,105) ;
\draw  [fill={rgb, 255:red, 0; green, 0; blue, 0 }  ,fill opacity=1 ] (340,35.5) .. controls (340,34.12) and (338.88,33) .. (337.5,33) .. controls (336.12,33) and (335,34.12) .. (335,35.5) .. controls (335,36.88) and (336.12,38) .. (337.5,38) .. controls (338.88,38) and (340,36.88) .. (340,35.5) -- cycle ;
\draw  [fill={rgb, 255:red, 0; green, 0; blue, 0 }  ,fill opacity=1 ] (370,35.5) .. controls (370,34.12) and (368.88,33) .. (367.5,33) .. controls (366.12,33) and (365,34.12) .. (365,35.5) .. controls (365,36.88) and (366.12,38) .. (367.5,38) .. controls (368.88,38) and (370,36.88) .. (370,35.5) -- cycle ;
\draw    (337.5,35.5) -- (307.5,107.5) ;
\draw    (367.5,35.5) -- (337.5,107.5) ;
\draw    (337.5,35.5) -- (392,112) ;
\draw    (367.5,35.5) -- (399,112) ;
\draw  [fill={rgb, 255:red, 0; green, 0; blue, 0 }  ,fill opacity=1 ] (460,107.5) .. controls (460,106.12) and (458.88,105) .. (457.5,105) .. controls (456.12,105) and (455,106.12) .. (455,107.5) .. controls (455,108.88) and (456.12,110) .. (457.5,110) .. controls (458.88,110) and (460,108.88) .. (460,107.5) -- cycle ;
\draw  [fill={rgb, 255:red, 0; green, 0; blue, 0 }  ,fill opacity=1 ] (490,107.5) .. controls (490,106.12) and (488.88,105) .. (487.5,105) .. controls (486.12,105) and (485,106.12) .. (485,107.5) .. controls (485,108.88) and (486.12,110) .. (487.5,110) .. controls (488.88,110) and (490,108.88) .. (490,107.5) -- cycle ;
\draw    (420,30) -- (406,112) ;
\draw    (427,35) -- (457.5,107.5) ;
\draw    (434,35) -- (487.5,107.5) ;
\draw  [fill={rgb, 255:red, 0; green, 0; blue, 0 }  ,fill opacity=1 ] (490,35.5) .. controls (490,34.12) and (488.88,33) .. (487.5,33) .. controls (486.12,33) and (485,34.12) .. (485,35.5) .. controls (485,36.88) and (486.12,38) .. (487.5,38) .. controls (488.88,38) and (490,36.88) .. (490,35.5) -- cycle ;
\draw    (487.5,35.5) -- (457.5,107.5) ;
\draw    (487.5,35.5) -- (487.5,107.5) ;
\draw  [fill={rgb, 255:red, 0; green, 0; blue, 0 }  ,fill opacity=1 ] (250,35.5) .. controls (250,34.12) and (248.88,33) .. (247.5,33) .. controls (246.12,33) and (245,34.12) .. (245,35.5) .. controls (245,36.88) and (246.12,38) .. (247.5,38) .. controls (248.88,38) and (250,36.88) .. (250,35.5) -- cycle ;
\draw  [fill={rgb, 255:red, 0; green, 0; blue, 0 }  ,fill opacity=1 ] (250,107.5) .. controls (250,106.12) and (248.88,105) .. (247.5,105) .. controls (246.12,105) and (245,106.12) .. (245,107.5) .. controls (245,108.88) and (246.12,110) .. (247.5,110) .. controls (248.88,110) and (250,108.88) .. (250,107.5) -- cycle ;
\draw    (217.5,35.5) -- (250,35.5) ;
\draw    (247.5,35.5) -- (247.5,105) ;
\draw    (277.5,35.5) -- (247.5,107.5) ;
\draw    (247.5,33) -- (277.5,107.5) ;
\draw    (277.5,35.5) -- (307.5,107.5) ;

\draw (240,110) node [anchor=north west][inner sep=0.75pt]   [align=left] {$v_3$};
\draw (270,110) node [anchor=north west][inner sep=0.75pt]   [align=left] {$y_1$};
\draw (300,110) node [anchor=north west][inner sep=0.75pt]   [align=left] {$y_2$};
\draw (300,20) node [anchor=north west][inner sep=0.75pt]  [xslant=0.01] [align=left] {$x_2$};
\draw (330,20) node [anchor=north west][inner sep=0.75pt]  [xslant=0.01] [align=left] {$x_3$};
\draw (360,20) node [anchor=north west][inner sep=0.75pt]  [xslant=0.01] [align=left] {$x_4$};
\draw (330,110) node [anchor=north west][inner sep=0.75pt]   [align=left] {$y_3$};
\draw (390,110) node [anchor=north west][inner sep=0.75pt]   [align=left] {$\cdots$};
\draw (450,110) node [anchor=north west][inner sep=0.75pt]   [align=left] {$y_{k-2}$};
\draw (480,110) node [anchor=north west][inner sep=0.75pt]   [align=left] {$y_{k-1}$};
\draw (414,23) node [anchor=north west][inner sep=0.75pt]   [align=left] {$\cdots$};
\draw (470,20) node [anchor=north west][inner sep=0.75pt]   [align=left] {$x_{k-1}$};
\draw (183,88) node [anchor=north west][inner sep=0.75pt]   [align=left] {$S_2$};
\draw (501,30) node [anchor=north west][inner sep=0.75pt]   [align=left] {$L_1$};
\draw (505,88) node [anchor=north west][inner sep=0.75pt]   [align=left] {$L_2$};
\draw (341,142) node [anchor=north west][inner sep=0.75pt]   [align=left] {Figure 1};
\draw (207,20) node [anchor=north west][inner sep=0.75pt]  [xslant=0.01] [align=left] {$v_1$};
\draw (237,20) node [anchor=north west][inner sep=0.75pt]  [xslant=0.01] [align=left] {$v_2$};
\draw (267,20) node [anchor=north west][inner sep=0.75pt]  [xslant=0.01] [align=left] {$x_1$};
\draw (183,30) node [anchor=north west][inner sep=0.75pt]  [xslant=0.01] [align=left] {$S_1$};

\end{tikzpicture}

\end{center}
Now, we will consider the cases \ref{lab:2} and \ref{lab:3}. Let $\mathcal{N} = N(v_1) \cup N(v_2) \cup N(v_3)$.
\begin{claim}\label{small vertex}
$|\mathcal{N} \cap \mathcal{L}| \leq 1$.
\end{claim}
\begin{proof}
Suppose on the contrary that $|\mathcal{N} \cap \mathcal{L}| \geq 2$.
For \ref{lab:2}, we can only consider $v_2 \in S_1$, $v_3 \in S_1$ and $v_1 \in L_1$ because the proof of other case is similar. Suppose that there exists a vertex $x_1$ in $L_1$ adjacent to $v_2$. By Claim \ref{fact common}, we can find two different vertices $y_1$, $y_2$ in $L_2$ such that $y_1$ and $y_2$ are adjacent to $x_1$ and $v_1$ separately. we also can find two different vertices $x_2$, $x_3$ in $L_1$ such that $y_1$ and $y_2$ are adjacent to $x_2$ and $x_3$ separately. Do the operation repeatedly until the last step, we choose a common neighbour $x_{k-1}$ in $L_1$ of $y_{k-2}$ and $y_{k-1}$. Then $x_1y_1x_2 \cdots x_3y_2v_1v_3v_2$ is a copy of $C_{2k+1}$(see Figure 2), a contradiction.
\begin{center}

\tikzset{every picture/.style={line width=0.75pt}} 

\begin{tikzpicture}[x=0.75pt,y=0.75pt,yscale=-1,xscale=1]

\draw   (183,14.7) .. controls (183,10.17) and (186.67,6.5) .. (191.2,6.5) -- (523.8,6.5) .. controls (528.33,6.5) and (532,10.17) .. (532,14.7) -- (532,39.3) .. controls (532,43.83) and (528.33,47.5) .. (523.8,47.5) -- (191.2,47.5) .. controls (186.67,47.5) and (183,43.83) .. (183,39.3) -- cycle ;
\draw   (181.2,94) .. controls (181.2,89.58) and (184.78,86) .. (189.2,86) -- (522,86) .. controls (526.42,86) and (530,89.58) .. (530,94) -- (530,118) .. controls (530,122.42) and (526.42,126) .. (522,126) -- (189.2,126) .. controls (184.78,126) and (181.2,122.42) .. (181.2,118) -- cycle ;
\draw    (231,5.8) -- (231,47.8) ;
\draw    (231.2,86.6) -- (231.2,126.6) ;
\draw  [fill={rgb, 255:red, 0; green, 0; blue, 0 }  ,fill opacity=1 ] (225,11.5) .. controls (225,10.12) and (223.88,9) .. (222.5,9) .. controls (221.12,9) and (220,10.12) .. (220,11.5) .. controls (220,12.88) and (221.12,14) .. (222.5,14) .. controls (223.88,14) and (225,12.88) .. (225,11.5) -- cycle ;
\draw  [fill={rgb, 255:red, 0; green, 0; blue, 0 }  ,fill opacity=1 ] (280,107.5) .. controls (280,106.12) and (278.88,105) .. (277.5,105) .. controls (276.12,105) and (275,106.12) .. (275,107.5) .. controls (275,108.88) and (276.12,110) .. (277.5,110) .. controls (278.88,110) and (280,108.88) .. (280,107.5) -- cycle ;
\draw  [fill={rgb, 255:red, 0; green, 0; blue, 0 }  ,fill opacity=1 ] (310,107.5) .. controls (310,106.12) and (308.88,105) .. (307.5,105) .. controls (306.12,105) and (305,106.12) .. (305,107.5) .. controls (305,108.88) and (306.12,110) .. (307.5,110) .. controls (308.88,110) and (310,108.88) .. (310,107.5) -- cycle ;
\draw  [fill={rgb, 255:red, 0; green, 0; blue, 0 }  ,fill opacity=1 ] (340,107.5) .. controls (340,106.12) and (338.88,105) .. (337.5,105) .. controls (336.12,105) and (335,106.12) .. (335,107.5) .. controls (335,108.88) and (336.12,110) .. (337.5,110) .. controls (338.88,110) and (340,108.88) .. (340,107.5) -- cycle ;
\draw  [fill={rgb, 255:red, 0; green, 0; blue, 0 }  ,fill opacity=1 ] (280,35.5) .. controls (280,34.12) and (278.88,33) .. (277.5,33) .. controls (276.12,33) and (275,34.12) .. (275,35.5) .. controls (275,36.88) and (276.12,38) .. (277.5,38) .. controls (278.88,38) and (280,36.88) .. (280,35.5) -- cycle ;
\draw  [fill={rgb, 255:red, 0; green, 0; blue, 0 }  ,fill opacity=1 ] (310,35.5) .. controls (310,34.12) and (308.88,33) .. (307.5,33) .. controls (306.12,33) and (305,34.12) .. (305,35.5) .. controls (305,36.88) and (306.12,38) .. (307.5,38) .. controls (308.88,38) and (310,36.88) .. (310,35.5) -- cycle ;
\draw    (307.5,35.5) -- (277.5,107.5) ;
\draw    (307.5,35.5) -- (337.5,105) ;
\draw  [fill={rgb, 255:red, 0; green, 0; blue, 0 }  ,fill opacity=1 ] (340,35.5) .. controls (340,34.12) and (338.88,33) .. (337.5,33) .. controls (336.12,33) and (335,34.12) .. (335,35.5) .. controls (335,36.88) and (336.12,38) .. (337.5,38) .. controls (338.88,38) and (340,36.88) .. (340,35.5) -- cycle ;
\draw  [fill={rgb, 255:red, 0; green, 0; blue, 0 }  ,fill opacity=1 ] (225,35.5) .. controls (225,34.12) and (223.88,33) .. (222.5,33) .. controls (221.12,33) and (220,34.12) .. (220,35.5) .. controls (220,36.88) and (221.12,38) .. (222.5,38) .. controls (223.88,38) and (225,36.88) .. (225,35.5) -- cycle ;
\draw    (337.5,35.5) -- (307.5,107.5) ;
\draw    (367.5,35.5) -- (337.5,107.5) ;
\draw    (337.5,35.5) -- (392,105) ;
\draw    (367.5,33) -- (399,105) ;
\draw  [fill={rgb, 255:red, 0; green, 0; blue, 0 }  ,fill opacity=1 ] (460,107.5) .. controls (460,106.12) and (458.88,105) .. (457.5,105) .. controls (456.12,105) and (455,106.12) .. (455,107.5) .. controls (455,108.88) and (456.12,110) .. (457.5,110) .. controls (458.88,110) and (460,108.88) .. (460,107.5) -- cycle ;
\draw  [fill={rgb, 255:red, 0; green, 0; blue, 0 }  ,fill opacity=1 ] (490,107.5) .. controls (490,106.12) and (488.88,105) .. (487.5,105) .. controls (486.12,105) and (485,106.12) .. (485,107.5) .. controls (485,108.88) and (486.12,110) .. (487.5,110) .. controls (488.88,110) and (490,108.88) .. (490,107.5) -- cycle ;
\draw    (420,35) -- (406,105) ;
\draw    (427,35) -- (457.5,107.5) ;
\draw    (434,35) -- (487.5,107.5) ;
\draw  [fill={rgb, 255:red, 0; green, 0; blue, 0 }  ,fill opacity=1 ] (490,35.5) .. controls (490,34.12) and (488.88,33) .. (487.5,33) .. controls (486.12,33) and (485,34.12) .. (485,35.5) .. controls (485,36.88) and (486.12,38) .. (487.5,38) .. controls (488.88,38) and (490,36.88) .. (490,35.5) -- cycle ;
\draw    (487.5,35.5) -- (457.5,107.5) ;
\draw    (487.5,35.5) -- (487.5,107.5) ;
\draw  [fill={rgb, 255:red, 0; green, 0; blue, 0 }  ,fill opacity=1 ] (250,35.5) .. controls (250,34.12) and (248.88,33) .. (247.5,33) .. controls (246.12,33) and (245,34.12) .. (245,35.5) .. controls (245,36.88) and (246.12,38) .. (247.5,38) .. controls (248.88,38) and (250,36.88) .. (250,35.5) -- cycle ;
\draw  [fill={rgb, 255:red, 0; green, 0; blue, 0 }  ,fill opacity=1 ] (250,107.5) .. controls (250,106.12) and (248.88,105) .. (247.5,105) .. controls (246.12,105) and (245,106.12) .. (245,107.5) .. controls (245,108.88) and (246.12,110) .. (247.5,110) .. controls (248.88,110) and (250,108.88) .. (250,107.5) -- cycle ;
\draw    (222.5,11.5) -- (222.5,35.5) ;
\draw    (277.5,35.5) -- (247.5,107.5) ;
\draw    (247.5,35.5) -- (247.5,107.5) ;
\draw    (277.5,35.5) -- (307.5,107.5) ;
\draw  [fill={rgb, 255:red, 0; green, 0; blue, 0 }  ,fill opacity=1 ] (264,11.5) .. controls (264,10.12) and (262.88,9) .. (261.5,9) .. controls (260.12,9) and (259,10.12) .. (259,11.5) .. controls (259,12.88) and (260.12,14) .. (261.5,14) .. controls (262.88,14) and (264,12.88) .. (264,11.5) -- cycle ;
\draw    (261.5,11.5) -- (222.5,11.5) ;
\draw    (222.5,35.5) -- (261.5,11.5) ;
\draw    (222.5,35.5) -- (247.5,35.5) ;
\draw    (261.5,11.5) -- (277.5,107.5) ;
\draw  [fill={rgb, 255:red, 0; green, 0; blue, 0 }  ,fill opacity=1 ] (370,35.5) .. controls (370,34.12) and (368.88,33) .. (367.5,33) .. controls (366.12,33) and (365,34.12) .. (365,35.5) .. controls (365,36.88) and (366.12,38) .. (367.5,38) .. controls (368.88,38) and (370,36.88) .. (370,35.5) -- cycle ;

\draw (240,110) node [anchor=north west][inner sep=0.75pt]   [align=left] {$y_1$};
\draw (270,110) node [anchor=north west][inner sep=0.75pt]   [align=left] {$y_2$};
\draw (300,110) node [anchor=north west][inner sep=0.75pt]   [align=left] {$y_3$};
\draw (298,20) node [anchor=north west][inner sep=0.75pt]  [xslant=0.01] [align=left] {$x_3$};
\draw (325,20) node [anchor=north west][inner sep=0.75pt]  [xslant=0.01] [align=left] {$x_4$};
\draw (364,20) node [anchor=north west][inner sep=0.75pt]  [xslant=0.01] [align=left] {$x_5$};
\draw (328,110) node [anchor=north west][inner sep=0.75pt]   [align=left] {$y_4$};
\draw (390,105) node [anchor=north west][inner sep=0.75pt]   [align=left] {$\cdots$};
\draw (450,110) node [anchor=north west][inner sep=0.75pt]   [align=left] {$y_{k-2}$};
\draw (480,110) node [anchor=north west][inner sep=0.75pt]   [align=left] {$y_{k-1}$};
\draw (415,22) node [anchor=north west][inner sep=0.75pt]   [align=left] {$\cdots$};
\draw (475,20) node [anchor=north west][inner sep=0.75pt]   [align=left] {$x_{k-1}$};
\draw (183,90) node [anchor=north west][inner sep=0.75pt]   [align=left] {$S_2$};
\draw (501,30) node [anchor=north west][inner sep=0.75pt]   [align=left] {$L_1$};
\draw (505,90) node [anchor=north west][inner sep=0.75pt]   [align=left] {$L_2$};
\draw (341,142) node [anchor=north west][inner sep=0.75pt]   [align=left] {Figure 2};
\draw (205,8) node [anchor=north west][inner sep=0.75pt]  [xslant=0.01] [align=left] {$v_3$};
\draw (240,20) node [anchor=north west][inner sep=0.75pt]  [xslant=0.01] [align=left] {$x_1$};
\draw (265,20) node [anchor=north west][inner sep=0.75pt]  [xslant=0.01] [align=left] {$x_2$};
\draw (183.15,30) node [anchor=north west][inner sep=0.75pt]  [xslant=0.01] [align=left] {$S_1$};
\draw (205,27) node [anchor=north west][inner sep=0.75pt]  [xslant=0.01] [align=left] {$v_2$};
\draw (266,8) node [anchor=north west][inner sep=0.75pt]  [xslant=0.01] [align=left] {$v_1$};

\end{tikzpicture}

\end{center}

For \ref{lab:3}, we can assume that $\{v_1, v_2, v_3\} \subseteq S_1$ because the proof of other case is similar. Suppose $v_1$, $v_2$ are adjacent to $x_1$, $x_2$ in $L_1$ separately. By Claim \ref{fact common}, we can find two different vertices $y_1$, $y_2$ in $L_2$ such that $y_1$ and $y_2$ are adjacent to $x_1$ and $x_2$ separately. Do the operation repeatedly and until the last step, we choose the common neighbour $y_{k-1}$ in $L_2$ of $x_{k-1}$ and $x_k$. Then $x_1y_1x_3y_3 \cdots y_2x_2v_2v_1$ is a copy of $C_{2k+1}$(see Figure 3), a contradiction.

\begin{center}

\tikzset{every picture/.style={line width=0.75pt}} 

\begin{tikzpicture}[x=0.75pt,y=0.75pt,yscale=-1,xscale=1]

\draw   (183,14.7) .. controls (183,10.17) and (186.67,6.5) .. (191.2,6.5) -- (523.8,6.5) .. controls (528.33,6.5) and (532,10.17) .. (532,14.7) -- (532,39.3) .. controls (532,43.83) and (528.33,47.5) .. (523.8,47.5) -- (191.2,47.5) .. controls (186.67,47.5) and (183,43.83) .. (183,39.3) -- cycle ;
\draw   (181.2,94) .. controls (181.2,89.58) and (184.78,86) .. (189.2,86) -- (522,86) .. controls (526.42,86) and (530,89.58) .. (530,94) -- (530,118) .. controls (530,122.42) and (526.42,126) .. (522,126) -- (189.2,126) .. controls (184.78,126) and (181.2,122.42) .. (181.2,118) -- cycle ;
\draw    (231,5.8) -- (231,47.8) ;
\draw    (231.2,86.6) -- (231.2,126.6) ;
\draw  [fill={rgb, 255:red, 0; green, 0; blue, 0 }  ,fill opacity=1 ] (225,11.5) .. controls (225,10.12) and (223.88,9) .. (222.5,9) .. controls (221.12,9) and (220,10.12) .. (220,11.5) .. controls (220,12.88) and (221.12,14) .. (222.5,14) .. controls (223.88,14) and (225,12.88) .. (225,11.5) -- cycle ;
\draw  [fill={rgb, 255:red, 0; green, 0; blue, 0 }  ,fill opacity=1 ] (280,107.5) .. controls (280,106.12) and (278.88,105) .. (277.5,105) .. controls (276.12,105) and (275,106.12) .. (275,107.5) .. controls (275,108.88) and (276.12,110) .. (277.5,110) .. controls (278.88,110) and (280,108.88) .. (280,107.5) -- cycle ;
\draw  [fill={rgb, 255:red, 0; green, 0; blue, 0 }  ,fill opacity=1 ] (310,107.5) .. controls (310,106.12) and (308.88,105) .. (307.5,105) .. controls (306.12,105) and (305,106.12) .. (305,107.5) .. controls (305,108.88) and (306.12,110) .. (307.5,110) .. controls (308.88,110) and (310,108.88) .. (310,107.5) -- cycle ;
\draw  [fill={rgb, 255:red, 0; green, 0; blue, 0 }  ,fill opacity=1 ] (340,107.5) .. controls (340,106.12) and (338.88,105) .. (337.5,105) .. controls (336.12,105) and (335,106.12) .. (335,107.5) .. controls (335,108.88) and (336.12,110) .. (337.5,110) .. controls (338.88,110) and (340,108.88) .. (340,107.5) -- cycle ;
\draw  [fill={rgb, 255:red, 0; green, 0; blue, 0 }  ,fill opacity=1 ] (280,35.5) .. controls (280,34.12) and (278.88,33) .. (277.5,33) .. controls (276.12,33) and (275,34.12) .. (275,35.5) .. controls (275,36.88) and (276.12,38) .. (277.5,38) .. controls (278.88,38) and (280,36.88) .. (280,35.5) -- cycle ;
\draw  [fill={rgb, 255:red, 0; green, 0; blue, 0 }  ,fill opacity=1 ] (310,35.5) .. controls (310,34.12) and (308.88,33) .. (307.5,33) .. controls (306.12,33) and (305,34.12) .. (305,35.5) .. controls (305,36.88) and (306.12,38) .. (307.5,38) .. controls (308.88,38) and (310,36.88) .. (310,35.5) -- cycle ;
\draw    (307.5,35.5) -- (277.5,107.5) ;
\draw    (307.5,35.5) -- (337.5,105) ;
\draw  [fill={rgb, 255:red, 0; green, 0; blue, 0 }  ,fill opacity=1 ] (340,35.5) .. controls (340,34.12) and (338.88,33) .. (337.5,33) .. controls (336.12,33) and (335,34.12) .. (335,35.5) .. controls (335,36.88) and (336.12,38) .. (337.5,38) .. controls (338.88,38) and (340,36.88) .. (340,35.5) -- cycle ;
\draw  [fill={rgb, 255:red, 0; green, 0; blue, 0 }  ,fill opacity=1 ] (225,35.5) .. controls (225,34.12) and (223.88,33) .. (222.5,33) .. controls (221.12,33) and (220,34.12) .. (220,35.5) .. controls (220,36.88) and (221.12,38) .. (222.5,38) .. controls (223.88,38) and (225,36.88) .. (225,35.5) -- cycle ;
\draw    (337.5,35.5) -- (307.5,107.5) ;
\draw    (367.5,35.5) -- (337.5,107.5) ;
\draw    (337.5,35.5) -- (392,105) ;
\draw    (367.5,33) -- (399,105) ;
\draw  [fill={rgb, 255:red, 0; green, 0; blue, 0 }  ,fill opacity=1 ] (462,38) .. controls (462,36.62) and (460.88,35.5) .. (459.5,35.5) .. controls (458.12,35.5) and (457,36.62) .. (457,38) .. controls (457,39.38) and (458.12,40.5) .. (459.5,40.5) .. controls (460.88,40.5) and (462,39.38) .. (462,38) -- cycle ;
\draw  [fill={rgb, 255:red, 0; green, 0; blue, 0 }  ,fill opacity=1 ] (490,107.5) .. controls (490,106.12) and (488.88,105) .. (487.5,105) .. controls (486.12,105) and (485,106.12) .. (485,107.5) .. controls (485,108.88) and (486.12,110) .. (487.5,110) .. controls (488.88,110) and (490,108.88) .. (490,107.5) -- cycle ;
\draw    (420,35) -- (406,105) ;
\draw    (459.5,38) -- (487.5,107.5) ;
\draw  [fill={rgb, 255:red, 0; green, 0; blue, 0 }  ,fill opacity=1 ] (490,35.5) .. controls (490,34.12) and (488.88,33) .. (487.5,33) .. controls (486.12,33) and (485,34.12) .. (485,35.5) .. controls (485,36.88) and (486.12,38) .. (487.5,38) .. controls (488.88,38) and (490,36.88) .. (490,35.5) -- cycle ;
\draw    (487.5,35.5) -- (487.5,107.5) ;
\draw  [fill={rgb, 255:red, 0; green, 0; blue, 0 }  ,fill opacity=1 ] (250,35.5) .. controls (250,34.12) and (248.88,33) .. (247.5,33) .. controls (246.12,33) and (245,34.12) .. (245,35.5) .. controls (245,36.88) and (246.12,38) .. (247.5,38) .. controls (248.88,38) and (250,36.88) .. (250,35.5) -- cycle ;
\draw  [fill={rgb, 255:red, 0; green, 0; blue, 0 }  ,fill opacity=1 ] (250,107.5) .. controls (250,106.12) and (248.88,105) .. (247.5,105) .. controls (246.12,105) and (245,106.12) .. (245,107.5) .. controls (245,108.88) and (246.12,110) .. (247.5,110) .. controls (248.88,110) and (250,108.88) .. (250,107.5) -- cycle ;
\draw    (222.5,11.5) -- (222.5,35.5) ;
\draw    (277.5,35.5) -- (247.5,107.5) ;
\draw    (247.5,35.5) -- (247.5,107.5) ;
\draw    (277.5,35.5) -- (307.5,107.5) ;
\draw  [fill={rgb, 255:red, 0; green, 0; blue, 0 }  ,fill opacity=1 ] (264,11.5) .. controls (264,10.12) and (262.88,9) .. (261.5,9) .. controls (260.12,9) and (259,10.12) .. (259,11.5) .. controls (259,12.88) and (260.12,14) .. (261.5,14) .. controls (262.88,14) and (264,12.88) .. (264,11.5) -- cycle ;
\draw    (261.5,11.5) -- (222.5,11.5) ;
\draw    (222.5,35.5) -- (247.5,35.5) ;
\draw    (261.5,11.5) -- (277.5,107.5) ;
\draw  [fill={rgb, 255:red, 0; green, 0; blue, 0 }  ,fill opacity=1 ] (370,35.5) .. controls (370,34.12) and (368.88,33) .. (367.5,33) .. controls (366.12,33) and (365,34.12) .. (365,35.5) .. controls (365,36.88) and (366.12,38) .. (367.5,38) .. controls (368.88,38) and (370,36.88) .. (370,35.5) -- cycle ;
\draw  [fill={rgb, 255:red, 0; green, 0; blue, 0 }  ,fill opacity=1 ] (201,23.5) .. controls (201,22.12) and (199.88,21) .. (198.5,21) .. controls (197.12,21) and (196,22.12) .. (196,23.5) .. controls (196,24.88) and (197.12,26) .. (198.5,26) .. controls (199.88,26) and (201,24.88) .. (201,23.5) -- cycle ;
\draw    (198.5,23.5) -- (222.5,11.5) ;
\draw    (198.5,23.5) -- (222.5,35.5) ;

\draw (238,109) node [anchor=north west][inner sep=0.75pt]   [align=left] {$y_1$};
\draw (270,109) node [anchor=north west][inner sep=0.75pt]   [align=left] {$y_2$};
\draw (296,109) node [anchor=north west][inner sep=0.75pt]   [align=left] {$y_3$};
\draw (297,20) node [anchor=north west][inner sep=0.75pt]  [xslant=0.01] [align=left] {$x_4$};
\draw (326,20) node [anchor=north west][inner sep=0.75pt]  [xslant=0.01] [align=left] {$x_5$};
\draw (363,20) node [anchor=north west][inner sep=0.75pt]  [xslant=0.01] [align=left] {$x_6$};
\draw (325,109) node [anchor=north west][inner sep=0.75pt]   [align=left] {$y_4$};
\draw (390,105) node [anchor=north west][inner sep=0.75pt]   [align=left] {$\cdots$};
\draw (446,20) node [anchor=north west][inner sep=0.75pt]   [align=left] {$x_{k-1}$};
\draw (480,109) node [anchor=north west][inner sep=0.75pt]   [align=left] {$y_{k-1}$};
\draw (409,23) node [anchor=north west][inner sep=0.75pt]   [align=left] {$\cdots$};
\draw (480,20) node [anchor=north west][inner sep=0.75pt]   [align=left] {$x_k$};
\draw (183,88) node [anchor=north west][inner sep=0.75pt]   [align=left] {$S_2$};
\draw (501,30) node [anchor=north west][inner sep=0.75pt]   [align=left] {$L_1$};
\draw (505,88) node [anchor=north west][inner sep=0.75pt]   [align=left] {$L_2$};
\draw (341,142) node [anchor=north west][inner sep=0.75pt]   [align=left] {Figure 3};
\draw (237,20) node [anchor=north west][inner sep=0.75pt]  [xslant=0.01] [align=left] {$x_1$};
\draw (266,20) node [anchor=north west][inner sep=0.75pt]  [xslant=0.01] [align=left] {$x_3$};
\draw (202,30) node [anchor=north west][inner sep=0.75pt]  [xslant=0.01] [align=left] {$v_1$};
\draw (202,8) node [anchor=north west][inner sep=0.75pt]  [xslant=0.01] [align=left] {$v_2$};
\draw (183,14) node [anchor=north west][inner sep=0.75pt]  [xslant=0.01] [align=left] {$v_3$};
\draw (183.15,30) node [anchor=north west][inner sep=0.75pt]  [xslant=0.01] [align=left] {$S_1$};
\draw (265,6) node [anchor=north west][inner sep=0.75pt]  [xslant=0.01] [align=left] {$x_2$};

\end{tikzpicture}

\end{center}
Thus, $|\mathcal{N} \cap \mathcal{L}| \leq 1$.
\end{proof}
By Claim \ref{small vertex}, there exists a vertex $v \in \{v_1, v_2, v_3\}$ such that $|N(v) \cap \mathcal{L}| \leq 1$. Since $|S_i| \leq \frac{|V_i|}{4k}$ for $i \in [2]$, we have $d(v) \leq |S_1| + |S_2| + 1 \leq \frac{n}{4k} +1 < \frac{n}{2k} -1$ for $n \geq 8k = n_2$. Thus, if there exists a triangle in $G_n$, then there exists at least one vertex $v \in G_n$ such that $d(v) < \frac{n}{2k} -1$. Because $G_n$ is $C_{2k+1}$-free, $G_n[N(v)]$ is $P_{2k}$-free which implies that the number of edges in $G_n[N(v)]$ is at most $ (k-1)d(v)$ by Theorem \ref{59EG}. If we delete $v$ from $G_n$, it will destroy at most $(k-1)d(v)$ triangles and delete $d(v)$ edges. Fix $n^\ast = \max\{n_1, n_2\}$ and set $G' = G_n - v$. For all $n \geq n^\ast$ we can obtain $f(n-1) \geq f(G') = e(G') + t(G') - \left\lfloor\frac{(n-1)^2}{4}\right\rfloor$. Hence,
\begin{align*}
f(n-1) - f(n)&\geq e(G') + t(G') - \left\lfloor\frac{(n-1)^2}{4}\right\rfloor - \left(e(G_n) + t(G_n) - \left\lfloor\frac{n^2}{4}\right\rfloor \right) \\
&\geq \frac{2n-2}{4} - kd(v) \\
&\geq \frac{2n-2}{4} - k\left(\frac{n}{2k}-1\right) \\
&> k-1.
\end{align*}
Let $n_0 = (n^\ast)^3$. Next we prove that $G_{n_0}$ contains no triangle. Suppose $G_{n_0}$ contains a triangle, then there exists a graph $G_{n_0-1}$ on $n_0-1$ vertices with $f(n_0-1) = f(G_{n_0-1})$ and $G_{n_0-1}$ contains a triangle. Otherwise, if for every graph $G_{n_0-1}$ on $n_0-1$ vertices with $f(n_0-1) = f(G_{n_0-1})$ is triangle-free, then we have
\begin{align*}
  0 &\leq f(n_0) < f(n_0-1) = f(G_{n_0-1}) = e(G_{n_0-1}) - \left\lfloor\frac{(n_0-1)^2}{4}\right\rfloor=0,
\end{align*}
a contradiction. Thus, each graph in $\{G_{n_0}, G_{n_0-1}, \cdots, G_{n^\ast}\}$ contains a triangle. But by using the inequality \eqref{fn} iteratively, we can get
\begin{align*}
0 \leq f(n_0) &\leq f(n_0 -1) -(k-1) \leq f(n^\ast)- (n_0 - n^\ast)(k-1) \notag \\
& \leq \binom{n^\ast}{2} + \binom{n^\ast}{3} - n_0(k-1) + n^\ast(k-1) < 0,
\end{align*}
a contradiction. Therefore, $G_{n_0}$ is triangle-free.

Next, our goal is to prove that $G_n$ is triangle-free for every $n \geq n_0$. Otherwise, let $n_3$ be the first integer after $n_0$ such that $G_{n_3}$ contains a triangle. Then $G_{n_3-1}$ contains no triangle. But $0 \leq f(n_3) < f(n_3-1) - (k-1) < 0$ by \eqref{fn}, a contradiction. Thus $G_n$ must be triangle-free for each $n \geq n_0$. So $e(G_n) + t(G_n) = e(G_n) = \left\lfloor\frac{n^2}{4}\right\rfloor$ and $G_n = T_2(n)$ by Tur\'{a}n's theorem.
\end{proof}

\section{Proof of Theorem ~\ref{main-theorem}}\label{SEC:pf-main-thm}
We prove Theorem~\ref{main-theorem} in this section.
\begin{proof}[Proof of Theorem~\ref{main-theorem}]
Let $H = K_\ell \uproduct T_2(n-\ell)$. It is easy to see that $H$ is $(\ell+1)\cdot C_{2k+1}$-free and
\begin{align}\label{lower bound of odd cycle}
t(H) = \ell\left\lfloor\frac{(n-\ell)^2}{4}\right\rfloor + (n-\ell)\binom{\ell}{2} + \binom{\ell}{3}.
\end{align}
Let $G$ be a $(\ell+1) \cdot C_{2k+1}$-free maximal graph on $n$ vertices with $t(G)\ge t(H)$. Clearly, there exists a copy of $\ell \cdot C_{2k+1}$ in $G$. Let $U$ denote the vertex set of such a copy of $\ell \cdot C_{2k+1}$ and let $W = V(G) \setminus U$. Now, we define an auxiliary 3-graph $\mathcal{H}$ with $V(\mathcal{H}) = V(G)$ and
\begin{center}
$E(\mathcal{H}) := \left\{\{x, y, z\} \subseteq \binom{V}{3}:x, y, z\,\, \text{form a triangle in}\,\, G\right\}.$
\end{center}
We mention that the link graph $L_{\mathcal{H}}(v)\subset G$  for each $v\in V(\mathcal{H})$.

Let
\begin{align*}
L:=\left\{v \in V(\mathcal{H}): d_{{\mathcal{H}}}(v) \geq \frac{n^2}{8|U|}\right\}.
\end{align*}
In what follows, we show that the size of $L$ is exactly $\ell$.
\begin{claim} \label{upper bound L}
$|L| \leq \ell$.
\end{claim}
\begin{proof}
Suppose that $|L| \geq \ell+1$ and $\{v_1, \cdots, v_{\ell+1}\}\subseteq L$. By the definition of $L$, we obtain that $e(L_{\mathcal{H}}(v_i))=d_{\mathcal{H}}(v_i) \geq \frac{n^2}{8|U|}$ for each $i\in[\ell+1]$. By Theorem \ref{59EG}, we can find a copy of $P_{2k}$ in each $L_{\mathcal{H}}(v_i)$. Clearly, there is a choice such that this $\ell+1$ copies of $P_{2k}$ are pairwise disjoint as $e(L_{\mathcal{H}}(v_i))=\Theta(n^2)$ and $n$ is sufficiently large. Thus, we have a copy of $(\ell+1)\cdot P_{2k}$ in $G$ as $L_{\mathcal{H}}(v)\subset G$  for each $v\in V(\mathcal{H})$. This together with $\{v_1, \cdots, v_{\ell+1}\}$ forms a copy of $(\ell+1) \cdot C_{2k+1}$ in $G$,  a contradiction.
\end{proof}

Let $L'=L\cap U$. By Claim \ref{upper bound L}, we have $|L'|\le \ell$. Observe that $d_{\mathcal{H}}(v)<\frac{n^2}{8|U|}$ for each $v\in U\setminus L'$. It is easy to see that
\begin{align*}
t(G) &\leq |L'|e(W) +\binom{|L'|}{3} + \binom{|L'|}{2}|W| + \sum_{v \in U\setminus L'} d_{\mathcal{H}}(v) + \mathrm{ex}(|W|, C_3, C_{2k+1})\\
&\leq |L'|e(W) +(|U| - |L'|)\frac{n^2}{8|U|}+o(n^2)\\
&\leq|L'|e(W) +\frac{n^2}{8} + o(n^2).
\end{align*}
Note that $G[W]$ is $C_{2k+1}$-free and $e(W)\le n^2/4$ in view of Theorem \ref{15FG}. It follows that $t(G)<t(H)$ for sufficiently large $n$ providing that $|L|\leq \ell-1$. This leads to a contradiction. Hence, we conclude that $|L|=\ell$.

Write $S:= V(\mathcal{H})\setminus L$.

\begin{claim}\label{GS}
$G[S]$ is $C_{2k+1}$-free.
\end{claim}
\begin{proof}
Suppose on the contrary that there exists a copy of $C_{2k+1}$ in $G[S]$. Let $G' = G \setminus C_{2k+1}$ and let $\mathcal{H}'$ be a $3$-graph formed by triangles in $G'$. Clearly, we have $d_{\mathcal{H}'}(v) \geq \frac{n^2}{8|U|} - (2k+1)n=\Theta(n^2)$ for each $v\in L$. By Theorem \ref{59EG}, we can find a copy of $P_{2k}$ in $L_{\mathcal{H'}}(v)$ for each $v\in L$. Clearly, there is a choice such that this $\ell$ copies of $P_{2k}$ are pairwise disjoint as $e(L_{\mathcal{H'}}(v))=d_{\mathcal{H}'}(v)=\Theta(n^2)$ and $n$ is sufficiently large. Thus, we have a copy of $\ell \cdot P_{2k}$ in $G'$ as $L_{\mathcal{H'}}(v)\subset G'$  for each $v\in L$. This together with $L$ forms a copy of $\ell \cdot C_{2k+1}$ in $G'$. Thus, $G$ contains a copy of $(\ell+1)\cdot C_{2k+1}$,  a contradiction.
\end{proof}

By Theorem \ref{auxilary theorem} and Claim \ref{GS}, we obtain that $e(S) + t(G[S]) \leq \left\lfloor\frac{(n-\ell)^2}{4}\right\rfloor$. Thus
\begin{align*}
t(G) &\leq |L| \cdot e(S) + t(G[S]) + (n-\ell)\binom{\ell}{2} + \binom{\ell}{3}\\
&\leq (\ell-1)e(S) +  \left(e(S) + t(G[S])\right) + (n-\ell)\binom{\ell}{2}+ \binom{\ell}{3}\\
&\leq \ell\left\lfloor\frac{(n-\ell)^2}{4}\right\rfloor + (n-\ell)\binom{\ell}{2}+ \binom{\ell}{3},
\end{align*}
where the last inequality holds if and only if $G$ is isomorphic to $H$. This completes the proof of Theorem \ref{main-theorem}.
\end{proof}

$\textbf{Acknowledgements}$
The authors would like to thank Dr. Xizhi Liu for his fruitful discussions.

\bibliographystyle{abbrv}
\bibliography{CHBipartite}

\begin{thebibliography}{10}

\bibitem{AS16}
N.~Alon and C.~Shikhelman.
\newblock Many {$T$} copies in {$H$}-free graphs.
\newblock {\em J. Combin. Theory Ser. B}, 121:146--172, 2016.

\bibitem{BG08}
B.~Bollob\'{a}s and E.~Gy\H{o}ri.
\newblock Pentagons vs. triangles.
\newblock {\em Discrete Math.}, 308(19):4332--4336, 2008.

\bibitem{BS74}
J.~A. Bondy and M.~Simonovits.
\newblock Cycles of even length in graphs.
\newblock {\em J. Combinatorial Theory Ser. B}, 16:97--105, 1974.

\bibitem{BJ17}
B.~Bukh and Z.~Jiang.
\newblock A bound on the number of edges in graphs without an even cycle.
\newblock {\em Combin. Probab. Comput.}, 26(1):1--15, 2017.

\bibitem{Er62}
P.~Erd\H{o}s.
\newblock {E}xtremal problem in {G}raph theory.
\newblock {\em Arch. Math. (Basel)}, 13:222--227, 1962.

\bibitem{E62}
P.~Erd\H{o}s.
\newblock On the number of complete subgraphs contained in certain graphs.
\newblock {\em Magyar Tud. Akad. Mat. Kutat\'{o} Int. K\"{o}zl.}, 7:459--464,
  1962.

\bibitem{E84}
P.~Erd\H{o}s.
\newblock On some problems in graph theory, combinatorial analysis and
  combinatorial number theory.
\newblock In {\em Graph theory and combinatorics ({C}ambridge, 1983)}, pages
  1--17. Academic Press, London, 1984.

\bibitem{EG59}
P.~Erd\H{o}s and T.~Gallai.
\newblock On maximal paths and circuits of graphs.
\newblock {\em Acta Math. Acad. Sci. Hungar.}, 10:337--356 (unbound insert),
  1959.

\bibitem{FG15}
Z.~F\"{u}redi and D.~S. Gunderson.
\newblock Extremal numbers for odd cycles.
\newblock {\em Combin. Probab. Comput.}, 24(4):641--645, 2015.

\bibitem{FO17}
Z.~F\"{u}redi and L.~\"{O}zkahya.
\newblock On 3-uniform hypergraphs without a cycle of a given length.
\newblock {\em Discrete Appl. Math.}, 216:582--588, 2017.

\bibitem{G2023}
D.~Gerbner.
\newblock Generalized {T}ur\'{a}n problems for {$K_{2,t}$}.
\newblock {\em Electron. J. Combin.}, 30(1):Paper No. 1.34, 9, 2023.

\bibitem{GGMV20}
D.~Gerbner, E.~Gy\H{o}ri, A.~Methuku, and M.~Vizer.
\newblock Generalized {T}ur\'{a}n problems for even cycles.
\newblock {\em J. Combin. Theory Ser. B}, 145:169--213, 2020.

\bibitem{GMV19}
D.~Gerbner, A.~Methuku, and M.~Vizer.
\newblock Generalized {T}ur\'{a}n problems for disjoint copies of graphs.
\newblock {\em Discrete Math.}, 342(11):3130--3141, 2019.

\bibitem{G12}
A.~Grzesik.
\newblock On the maximum number of five-cycles in a triangle-free graph.
\newblock {\em J. Combin. Theory Ser. B}, 102(5):1061--1066, 2012.

\bibitem{GL13}
E.~Gy\H{o}ri and H.~Li.
\newblock The maximum number of triangles in {$C_{2k+1}$}-free graphs.
\newblock {\em Combin. Probab. Comput.}, 21(1-2):187--191, 2012.

\bibitem{HHK13}
H.~Hatami, J.~Hladk\'{y}, D.~Kr\'{a}\v{l}, S.~Norine, and A.~Razborov.
\newblock On the number of pentagons in triangle-free graphs.
\newblock {\em J. Combin. Theory Ser. A}, 120(3):722--732, 2013.

\bibitem{HHLYZ23}
J.~Hou, H.~Li, X.~Liu, L.-T. Yuan, and Y.~Zhang.
\newblock A step towards a general density {C}orr\'{a}di--{H}ajnal theorem.
\newblock {\em arXiv preprint arXiv:2302.09849}.

\bibitem{liu2023}
X.~Liu and J.~Song.
\newblock Exact results for some extremal problems on expansions i.
\newblock {\em arxiv preprint arXiv:2310.01736}.

\bibitem{MQ2020}
J.~Ma and Y.~Qiu.
\newblock Some sharp results on the generalized {T}ur\'{a}n numbers.
\newblock {\em European J. Combin.}, 84:103026, 16, 2020.

\bibitem{Moon68}
J.~W. Moon.
\newblock On independent complete subgraphs in a graph.
\newblock {\em Canadian J. Math.}, 20:95--102, 1968.

\bibitem{RS18}
A.~Roberts and A.~Scott.
\newblock Stability results for graphs with a critical edge.
\newblock {\em European J. Combin.}, 74:27--38, 2018.

\bibitem{41Turan}
P.~Tur{\'a}n.
\newblock On an extremal problem in graph theory.
\newblock {\em Mat. Fiz. Lapok}, 48:436--452, 1941.

\bibitem{ZCG23}
F.~Zhang, Y.~Chen, E.~Gyori, and X.~Zhu.
\newblock Maximum cliques in a graph without disjoint given subgraph.
\newblock {\em arxiv preprint arXiv:2309.09603}.

\bibitem{ZC22}
X.~Zhu and Y.~Chen.
\newblock Generalized {T}ur\'{a}n number for linear forests.
\newblock {\em Discrete Math.}, 345(10):Paper No. 112997, 12, 2022.

\bibitem{ZC23}
X.~Zhu, Y.~Chen, D.~Gerbner, E.~Gy\H{o}ri, and H.~H. Karim.
\newblock The maximum number of triangles in {$F_k$}-free graphs.
\newblock {\em European J. Combin.}, 114:Paper No. 103793, 19, 2023.

\end{thebibliography}
\end{document}